\documentclass[11pt,oneside,a4paper,abstract]{scrartcl}
\usepackage[english]{babel}
\usepackage[utf8]{inputenc}
\usepackage[T1]{fontenc}
\usepackage{lmodern}
\usepackage{amssymb}
\usepackage{amsmath}
\usepackage{amsthm}
\usepackage{amsfonts}
\usepackage{graphicx}
\usepackage{mathtools}
\usepackage[sort&compress,numbers]{natbib}
\usepackage{hyperref}
\hypersetup{
  pdfauthor={Johannes Siegele, Daniel F. Scharler, Hans-Peter Schröcker},
  pdftitle={Rational Motions with Generic Trajectories of Low Degree},
  pdfkeywords={},
  hidelinks,
}

\setkomafont{sectioning}{\bfseries\rmfamily}

\newtheorem{theo}{Theorem}

\newtheorem{cor}[theo]{Corollary}
\newtheorem{pro}[theo]{Proposition}				
\theoremstyle{definition}
\newtheorem{defi}[theo]{Definition}

\theoremstyle{remark}
\newtheorem{example}[theo]{Example}		
\newtheorem{rem}[theo]{Remark}
	
\newcommand{\D}{\mathbb D}
\newcommand{\K}{\mathbb K}
\newcommand{\R}{\mathbb R}
\newcommand{\N}{\mathcal N}
\newcommand{\C}{\mathbb C}
\newcommand{\Q}{\mathbb H}
\newcommand{\DQ}{\mathbb{DH}}
\newcommand{\CQ}{\mathbb{CH}}
\newcommand{\Pj}{\mathbb{P}}
\newcommand{\EG}{\operatorname{E}}
\newcommand{\qi}{\mathbf{i}}
\newcommand{\qj}{\mathbf{j}}
\newcommand{\qk}{\mathbf{k}}
\newcommand{\Cj}[1]{{#1}^\ast}
\newcommand{\norm}[1]{\left\|#1\right\|}

\newcommand{\SO}[1][3]{\operatorname{SO}(#1)}
\newcommand{\SE}[1][3]{\operatorname{SE}(#1)}
\newcommand{\St}{\mathcal{S}}
\newcommand{\I}{\operatorname{i}}
\newcommand{\EQ}{\mathcal{E}}
\newcommand{\eps}{\varepsilon}
\newcommand{\RLR}{\mathcal{L}}
\DeclareMathOperator{\mrpf}{mrpf}

\begin{document}
\title{Rational Motions with Generic Trajectories of Low Degree}
\author{
  Johannes Siegele\quad
  Daniel F. Scharler\quad
  Hans-Peter Schröcker\\
  University of Innsbruck, Unit Geometry and Surveying}

\maketitle

\begin{abstract}
    The trajectories of a rational motion given by a polynomial of degree $n$ in
  the dual quaternion model of rigid body displacements are generically of
  degree $2n$. In this article we study those exceptional motions whose
  trajectory degree is lower. An algebraic criterion for this drop of degree is
  existence of certain right factors, a geometric criterion involves one of two
  families of rulings on an invariant quadric. Our characterizations allow the
  systematic construction of rational motions with exceptional degree reduction
  and explain why the trajectory degrees of a rational motion and its inverse
  motion can be different.
 \end{abstract}


\section{Introduction}
\label{sec_introduction}

Rational motions are of particular interest in modern kinematics and robotics.
The rationality of the trajectories yields multiple algorithmic and numerical
benefits, see e.\,g., \cite{juettler93,roeschel98}. Rational motions are often
represented by homogeneous transformation matrices of dimension four by four
whose entries are rational functions. Our study is based on the dual quaternion
model of $\SE$ where rational motions appear as rational curves on the Study
quadric $\St$ and are parameterized by certain polynomials with dual quaternion
coefficients.

The \emph{degree} of a rational motion is the maximal degree of a trajectory
which, at the same time, is the degree of a generic trajectory. It coincides
with the degree of the motion as rational curve in the matrix model but not with
the degree as rational curve in the dual quaternion model. In fact, if the
rational curve in the dual quaternion model is given by a polynomial of degree
$n$, the motion degree is generically $2n$. However, exceptions to this relation
of degrees do exist. The most famous example is probably the Darboux motion, see
e.\,g., \cite[Chapter~9, \S~3]{bottema90} or \cite{li15}. It is represented by a
polynomial $C$ of degree three in dual quaternions while its generic
trajectories are of degree two.

Whenever the generic trajectory degree is less than $2n$, we speak of a
\emph{degree reduction.} It is well-known that a degree reduction is related to
the existence of intersection points of the rational motion with a certain
subspace of the Study quadric $\St$, the \emph{exceptional generator} $\EG$ (the
projective space over the vector space of non-invertible dual quaternions), see
e.\,g., \cite[Chapter~11]{bottema90} for planar motions or \cite{hamann11} for
certain line symmetric motions. The proofs of these papers easily generalize to
arbitrary rational motions, c.\,f. our Proposition~\ref{pro:degree}.

However, not all observed phenomena in this context can be explained by the
number of intersection points. In particular, there are rational motions of
degree $n$ that intersect $\EG$ in $m$ points but, nonetheless, have
trajectories of degree strictly less than $2n-m$. Again, an example is provided
by the Darboux motion where we have $n = 3$, $m = 2$ but trajectories of degree
$2 < 2n-m = 4$.

A second issue has been pointed out by Jon M. Selig in private communication. It
refers to the degrees of a rational motion and its inverse which is obtained by
interchanging the moving and the fixed frame. It is an important and natural
concept in mechanism science in situations where relative motions of links are
studied. If a rational motion is given by a dual quaternion polynomial $C$, its
inverse motion is given by the conjugate dual quaternion polynomial $\Cj{C}$
which apparently does not significantly differ from $C$ in its algebraic and
geometric characteristics. In particular, the curves parameterized by $C$ and
$\Cj{C}$ intersect the exceptional generator in the same number of points.
However, there are rational motions where the trajectory degrees of $C$ and
$\Cj{C}$ differ. An example is illustrated in Figure~\ref{fig:cardan-oldham}.
The elliptic or Cardan motion \cite[pp.~346--348]{bottema90} in the top row is
characterized by having two non-parallel straight line trajectories. Its generic
trajectories are ellipses, that is, rational curves of degree two. The inverse
motion (Figure~\ref{fig:cardan-oldham}, bottom) is called cardioid motion in
\cite[pp.~348--349]{bottema90} but is also known as Oldham motion. Two rigidly
connected lines move such that each line always passes through a fixed point.
Its generic trajectories, limaçons of Pascal, are of degree four. Cardan and
Oldham motion can be seen as special case of a Darboux motion and its inverse
motion. In fact, also for a Darboux motion the trajectories are of degree two
while the trajectories of the inverse motion are of degree four.

\begin{figure}
  \centering
  \includegraphics[page=1]{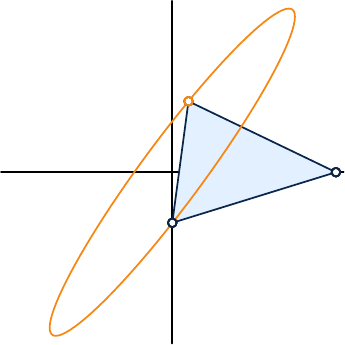}\hfill \includegraphics[page=3]{ellipticmotion}\hfill \includegraphics[page=5]{ellipticmotion}\hfill \includegraphics[page=7]{ellipticmotion}\\[2ex]
  \includegraphics[page=1]{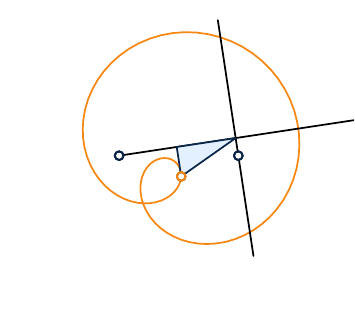}\hfill \includegraphics[page=3]{oldhammotion}\hfill \includegraphics[page=5]{oldhammotion}\hfill \includegraphics[page=7]{oldhammotion}
  \caption{Cardan motion (top) and Oldham motion (bottom) are inverse but of
    different degrees.}
  \label{fig:cardan-oldham}
\end{figure}

There is a vague understanding in the kinematics community that an
``exceptional'' degree reduction can occur if intersection points of $C$ and
$\EG$ lie on a certain quadric $\EQ$ of full rank and signature zero contained
in $\EG$. However, all these geometric concepts ($\EG$ and also $\EQ$) are
invariant with respect to conjugation and thus cannot explain the difference in
degree of the trajectories of $C$ and~$\Cj{C}$.

This article will provide a comprehensive answer to the questions above. A first
step in this direction was done in \cite{rad18} where the authors provided a
complete geometric characterization of all transformations of the
seven-dimensional projective space $\Pj^7$ over the vector space of dual
quaternions that are induced by coordinate changes in fixed and moving frame.
These projective transformations not only fix the Study quadric $\St$, the exceptional
generator $\EG$, and the quadric $\EQ$ in $\EG$ but also the two families of
(complex) rulings on $\EQ$. Since conjugation interchanges these families,
different trajectories of motion $C$ and inverse motion $\Cj{C}$ may be
explained by the position of intersection points of $C$ with $\EQ$ with respect
to one family of rulings. This turns out to be the case and is a new result for
rational motions in $\SE$. In case of the planar motion group $\SE[2]$, some
statements on motions with algebraic trajectories and their degrees are provided
in \cite[Chapter~11, \S~4]{bottema90}. They agree with the specialization of our
results to planar kinematics.

We organize this paper as follows. In Section~\ref{chap1} we provide some basic
information on quaternions and dual quaternions and their relation to
kinematics. We introduce complex quaternions which are needed to properly study
the quadric $\EQ$, and we explain fundamental geometric notions related to the
dual quaternion parametrization of $\SE$. Finally, we give precise definitions
for concepts related to rational parametrizations, rational curves, and motion
polynomials.

Our main results are proved in Section~\ref{chap2}. We investigate in more
detail the degree of trajectories and derive an algebraic criterion for
exceptionally low trajectory degrees in Section~\ref{sec2_1}. This algebraic
criterion is then used for a geometric characterization in Section~\ref{sec2_2}.
We also relate our findings to some results known from literature and talk about
the construction of motions with exceptionally low trajectory degree.

\section{Preliminaries}\label{chap1}

In this section we collect some properties of quaternions, complex quaternions
and dual quaternions and settle basic notation. We also introduce polynomials
with quaternion coefficients, relate them to rational motions and study some of
their fundamental properties.

\subsection{Quaternions, Complex Quaternions and Dual Quaternions}\label{sec1_1}

The algebra of \emph{quaternions $\Q$} is the real associative algebra generated
by the base $(1,\qi,\qj,\qk)$. The non-commutative multiplication of these units
is defined by
\begin{align*}
  \qi^2=\qj^2=\qk^2=\qi\qj\qk=-1.
\end{align*}
We sometimes refer to $\Q$ as algebra of Hamiltonian quaternions in order to
emphasize the distinction to the algebras of complex or dual quaternions.

The complex algebra generated by the same base and with the same rules
for multiplication is denoted by $\CQ$, elements of this algebra are called
\emph{complex quaternions.} A quaternion or complex quaternion $q$ is given by
\begin{equation}\label{quaternion}
  q=q_0+q_1\qi+q_2\qj+q_3\qk
\end{equation}
where $q_0$, $q_1$, $q_2$, $q_3$ are real or complex coefficients, respectively.
The imaginary unit of the complex numbers will be denoted by $\I$ and has to be
distinguished from the quaternion unit $\qi$. The (complex) quaternion conjugate
is defined as $\Cj{q}=q_0-q_1\qi-q_2\qj-q_3\qk$, the (complex) quaternion norm
is $\norm{q}=q\Cj{q}=q_0^2+q_1^2+q_2^2+q_3^2$. Note that this is not a norm in
the classical sense. Whenever the norm of a (complex) quaternion $q$ is
non-zero, the inverse of $q$ can be computed via $\Cj{q}/\norm{q}$. Otherwise
$q$ is a zero divisor. The only non-invertible Hamiltonian quaternion is zero
whence $\Q$ is a division ring.

The algebra of dual quaternions $\DQ$ is obtained by extending the real
coefficients in Equation~\eqref{quaternion} to dual numbers
$\mathbb{D}=\mathbb{R}[\eps]/\langle \eps^2 \rangle$. Any dual quaternion can be
written as $p+\eps d$ where $p\in\Q$ is called the \emph{primal} and $d\in\Q$
the \emph{dual} part. Multiplication follows the rules of quaternion
multiplication with the additional rules $\eps^2=0$, $\eps\qi = \qi\eps$,
$\eps\qj = \qj\eps$, and $\eps\qk = \qk\eps$. The conjugation of a dual
quaternion $p+\eps d$ is defined by $\Cj{(p+\eps d)}=\Cj{p}+\eps\Cj{d}$. The
norm $\norm{p+\eps d} \coloneqq(p+\eps d)\Cj{(p+\eps d)} = p\Cj{p} +
\eps(p\Cj{d}+d\Cj{p})$ is in general a dual number. It is real if and only if
the \emph{Study condition}
\begin{equation}
  \label{eq_study_condition}
  p\Cj{d} + d\Cj{p} = 0
\end{equation}
holds.

Quaternions, dual quaternions and complex quaternions also form a real or
complex vector space. By $\Pj(\Q)$, $\Pj(\DQ)$, and $\Pj(\CQ)$ we denote the
projective spaces over the respective vector space. Note that $\Pj(\Q)$ and
$\Pj(\DQ)$ are real projective spaces and $\Pj(\CQ)$ is the complex extension of
$\Pj(\Q)$. The elements of these projective spaces are classes of proportional
non-zero vectors. The projective point represented by $q \neq 0$ is denoted by
$[q]$. The symbol ``$\vee$'' indicates projective span. If the two points $[p]$,
$[q]\in\Pj(\CQ)$ are different, their connecting line is $[p] \vee [q]$. If they
are equal, we have $[p] = [q] = [p] \vee [q]$.

The Study condition \eqref{eq_study_condition} in the projective context is the
vanishing condition of the quadratic form $p + \eps d \mapsto p\Cj{d} +
d\Cj{p}$. The associated bilinear form hence defines a quadric $\St$ in
$\Pj(\DQ)$, the \emph{Study quadric,} which is of great importance in spacial
kinematics \cite[Chapter~11]{selig05}.

In contrast to the quaternions $\Q$, the complex quaternions $\CQ$ have
non-trivial zero-divisors. They are related to another important quadric. Zero
divisors are precisely the zeros of the quadratic form $q \mapsto q\Cj{q}$ whose
associated symmetric bilinear form defines a regular quadric $\N$ in $\Pj(\CQ)$
which we call the \emph{null quadric.} It is foliated by two families of
straight lines (rulings) which we call the \emph{left} and the \emph{right
  rulings.} More precisely, given $[p] \in \N$, the left ruling $L$ and the
right ruling $R$ through $[p]$ are the point sets
\begin{equation}
  \label{eq_rulings}
  L \coloneqq \{[q] \in \Pj(\CQ) \colon q\Cj{p}=0\},\quad
  R \coloneqq \{[q] \in \Pj(\CQ) \colon \Cj{p}q=0\},
\end{equation}
respectively. This requires some justification which we give in the next
proposition.

\begin{pro}
  \label{lem_rulings}
  The sets $L$ and $R$ in \eqref{eq_rulings} are two different straight lines
  through $[p]$ and contained in~$\N$.
\end{pro}

\begin{proof}
  Let $p\in\CQ$ be a complex quaternion such that $[p]\in\N$, i.\,e., $p\neq 0$
  and $\norm{p}=0$, and consider the linear map
  \begin{equation*}
    \varphi\colon \CQ \to \CQ,\quad
    q \mapsto q\Cj{p}.
  \end{equation*}
  The projective space generated by the kernel of this map is precisely $L$. It
  obviously contains
  \begin{equation}
    \begin{aligned}\label{quat_base}
      p&=\phantom{-}p_0+p_1\qi+p_2\qj+p_3\qk,\\
      \qi p&=-p_1+p_0\qi-p_3\qj+p_2\qk,\\
      \qj p&=-p_2+p_3\qi+p_0\qj-p_1\qk,\\
      \qk p&=-p_3-p_2\qi+p_1\qj+p_0\qk.
    \end{aligned}
  \end{equation}
  The vector space $V$ spanned by these four complex quaternions is a subspace
  of $\ker(\varphi)$. We will show that it is of (complex) dimension two. Assume
  at first that $\dim V \le 1$. Then there exist $\alpha_1$, $\alpha_2$,
  $\alpha_3\in\C\backslash\{0\}$ such that $p=\alpha_1 \qi p=\alpha_2\qj
  p=\alpha_3\qk p$. Comparing coefficients yields the conditions
  \begin{equation}\label{eq1_lem1.2}
      p_0 = -\alpha_i p_i,\quad p_i =\alpha_i p_0\quad\text{for $i \in \{1,2,3\}$}
  \end{equation}
  whence $\alpha_i^2=-1$. Consequently we get $\norm{p}=-2p_0^2$ and, because
  the norm of $p$ is assumed to be zero, it holds $p_0=0$. This together with
  Equation~\eqref{eq1_lem1.2} implies that $p=0$ which is a contradiction. Thus,
  the kernel of $\varphi$ is at least of affine dimension two and $L$ is at
  least a projective line. Since the maximal subspaces on $\N$ are lines, it
  suffices to show that $L$ is contained in $\N$: Let $[q]\in L$ be arbitrary.
  We have $0=q\Cj{p}$. Left-multiplying both sides of this equation with
  $\Cj{q}$ yields $0=\norm{q}\Cj{p}$. Since the norm of $q$ is a complex number
  and $p$ is not zero by assumption, it follows $\norm{q}=0$, hence
  $[q]\in\mathcal{N}$.

  In the same way it can be shown that $R$ is a ruling on $\mathcal{N}$.
  Finally, the two lines are different since the linear equations defining them
  are not equivalent.
\end{proof}

To conclude this subsection, we state a simple, yet important, corollary to
Proposition~\ref{lem_rulings} whose proof is left to the reader:

\begin{cor}\label{cor2}
  Let $[p]\in\N$ and $q\in\CQ$.
  \begin{enumerate}
  \item[(i)] If $qp\neq 0$ and $[qp]\neq[p]$, then $[p]$ and $[qp]$ span a left
    ruling of $\mathcal{N}$.
  \item[(ii)] If $pq\neq 0$ and $[pq]\neq[p]$, then $[p]$ and $[pq]$ span a
    right ruling of $\mathcal{N}$.
  \end{enumerate}
\end{cor}

\subsection{Quaternion Polynomials and Rational Motions}

For polynomials with coefficients in a ring, different notions of multiplication
are conceivable. In this article, we use polynomials with coefficients from
$\Q$, $\CQ$, or $\DQ$ to describe rational motions. The polynomial indeterminate
$t$ serves as a scalar motion parameter whence it is natural to assume that $t$
commutes with all coefficients. This defines a non-commutative multiplication of
polynomials and we denote the thus obtained polynomial rings by $\Q[t]$,
$\CQ[t]$, and $\DQ[t]$, respectively. The \emph{conjugate} $\Cj{P}$ of a
polynomial $P$ is defined as the polynomial obtained by conjugating its
coefficients. The \emph{norm polynomial} is defined as $\norm{P} \coloneqq
P\Cj{P}$. It is an element of $\R[t]$, $\C[t]$ or $\D[t]$, respectively.

Polynomials over rings come with the notions of left/right evaluation,
left/right zeros, and left/right factors. Given $P=\sum_{\ell=0}^{n}p_\ell
t^\ell$ in $\Q[t]$ (or $\CQ[t]$, $\DQ[t]$) we define the \emph{right
  evaluation} of $P$ at $q\in\Q$ (or $\CQ$, $\DQ$) to be
$\sum_{\ell=0}^np_\ell q^\ell$. The name ``right evaluation'' comes from the
fact that the indeterminate is written to the right of the coefficients before
being substituted for $q$. Quaternions where the right evaluation of a
polynomial $P$ vanishes are called \emph{right zeros} of $P$. The notions of
``left evaluation'' and ``left zero'' are similar but we will not need them in
the following.

A polynomial $S$ is called \emph{right factor} of $P$ if there exists a
polynomial $F$ such that $P = FS$ and it is called a \emph{left factor} if there
exists a polynomial $G$ such that $P = SG$.

There is a well-known double cover of $\SE$, the group of rigid body
displacements, by the group of dual quaternions of unit norm. The dual
quaternion $p + \eps d$ with $\norm{p + \eps d} = 1$ is sent to the map
\begin{equation}
  x = x_1\qi + x_2\qj + x_3\qk \in \R^3 \mapsto px\Cj{p} + 2p\Cj{d}
\end{equation}
which is also the image of $-(p + \eps d)$. In order to get rid of this
ambiguity, we prefer a projective formulation of this homomorphism. We embed the
vector space $\R^4$ into $\Q$ via $(x_0,x_1,x_2,x_3) \hookrightarrow x_0 +
x_1\qi + x_2\qj +x_3\qk$ and consider the real projective space $\Pj^3(\R)$ over
this vector space. Denote by $\R^\times$ the real multiplicative group. The
factor group
\begin{equation*}
  \DQ^\times \coloneqq \{ p + \eps d \in \DQ \colon \norm{p + \eps d} \in \R^\times \} / \R^\times
\end{equation*}
consists of points on the Study quadric $\St \subset \Pj(\DQ)$ minus the
\emph{exceptional generator $\EG \subset \Pj(\DQ)$} whose points are
characterized by having zero primal part $\EG = \{ [\eps d] \in \Pj(\DQ)\colon d
\in \Q \}$. This yields an isomorphism from $\DQ^\times$ to $\SE$ which sends
the point $[p + \eps d]$ to the map
\begin{equation}
  \label{eq_isomorphism}
  [x_0 + x_1\qi + x_2\qj + x_3\qk] \in \Pj^3 \mapsto [p(x_0 + x_1\qi + x_2\qj + x_3\qk)\Cj{p} + 2x_0p\Cj{d}].
\end{equation}
By projection on the primal part ($d = 0$) we obtain an isomorphism between
$\Pj(\Q) = (\Q \setminus \{0\}) / \R^\times$ and $\SO$. By extension of scalars
from $\R$ to $\C$ we define an action of points of $\Pj(\CQ)$ on points of
complex projective three space~$\Pj^3(\C)$. It is precisely the zero-divisors
that give singular maps of $\Pj^3(\C)$.

Rational motions are obtained by replacing $p$ and $d$ in \eqref{eq_isomorphism}
with (Hamiltonian or complex) quaternion polynomials $P$ and $D$, respectively.
Then the image of $[x_0 + x] \in \Pj^3$ is no longer a single point but
again a quaternion polynomial which we may regard as a parametric expression of
a rational curve. These concepts are central for this article so that we take
some care to capture all subtleties in their definitions.

Denote by $\K(t)$ the field of rational functions in the indeterminate $t$ over
the field $\K$ with $\K = \R$ or $\K = \C$ and by $\K^{m+1}(t)$ the vector space
of $(m+1)$-tuples over $\K(t)$. The projective space over this vector space is
denoted by $\Pj^m(\K(t))$.

\begin{defi}
  \label{def:rational-curve}
  A \emph{rational parametric expression $C$} is an element of $\K^{m+1}(t)$, a
  \emph{rational curve $[C]$} is an element of $\Pj^m(\K(t))$.
\end{defi}

The \emph{degree $\deg C$} of a rational parametric expression $C$ is defined as
the maximal degree of all coefficient functions. Given a rational curve $[C] \in
\Pj^m(\K(t))$ there exists a rational parametric expression $\hat C$ such that $[C] = [\hat
C]$ and the entries of $\hat C$ are polynomials. It is found by multiplying away
the denominators of all coefficients of $C$. We will generally consider
only polynomial parametric expressions. Among the many possible polynomial
parametric expressions of $[C]$ the ones of minimal degree are distinguished. They
are found by dividing by the greatest common divisor of all entries of $\hat C$
and are unique up to multiplication with a scalar. We call them \emph{reduced.}
The \emph{degree $\deg([C])$} of the rational curve $[C]$ is defined as degree
of any reduced representation.

The value of a rational curve at a scalar $t_0 \in \K$ is defined as the point
$[C(t_0)] \in \Pj^m(\K)$ for any representing parametric expression where $C(t_0)$
is well-defined an different from $0$. Note that for any $C$ there exist at most
finitely many values $t_0 \in \K$ that are not suitable for evaluation. Moreover,
$C_0(t_0)$ is always well-defined and different from zero for any reduced
parametric expression $C_0$. Finally, we have of course $[C(t_0)] =
[\tilde{C}(t_0)]$, whenever $[C] = [\tilde{C}]$ and these expressions are
well-defined.

The Zariski closure of the point set
\begin{equation*}
  \{ [C(t_0)] \colon t_0 \in \K \}
\end{equation*}
is an algebraic curve that, in general, contains one additional point. This
point is obtained as limit of $[C(t_0)]$ for $t_0 \to \infty$. More precisely,
we define
\begin{equation*}
  [C(\infty)] \coloneqq [\lim_{t_0 \to \infty} t_0^{-n} C_0(t_0)]
\end{equation*}
where $C_0$ is a reduced parametric expression and $n = \deg [C] = \deg C_0$. The
point $[C(\infty)]$ can also be thought of as being represented by the vector of
leading coefficients of~$C_0$.

Neither the degree nor the point set obtained by evaluating at scalars or
$\infty$ changes under fractional linear parameter transformation $t \mapsto
(\alpha t+\beta)/(\gamma t+\delta)$ with scalars $\alpha$, $\beta$, $\gamma$ and
$\delta$ that satisfy the regularity condition $\alpha\delta - \beta\gamma \neq
0$. These parameter transformations naturally include the value $t = \infty$
which is mapped to $\alpha/\gamma$ and is the pre-image of $-\delta/\gamma$
(with the conventions that $\alpha/\gamma = -\delta/\gamma = \infty$ if $\gamma
= 0$).

\begin{defi}
  \label{def:motion-polynomial}
  The (Hamiltonian, complex or dual) quaternion polynomial $P + \eps D$ is
  called a \emph{motion polynomial} if the polynomial Study condition $P\Cj{D} +
  D\Cj{P} = 0$ is satisfied and the norm polynomial $P\Cj{P}$ does not vanish.
\end{defi}

Any motion polynomial is a rational parametric expression in the sense of
Definition~\ref{def:rational-curve}. The corresponding rational curve is called
a \emph{rational motion.} For Hamiltonian and complex quaternions we have $D =
0$ and the Study condition in Definition~\ref{def:motion-polynomial} becomes
void. Moreover, $P\Cj{P} \neq 0$ for all $P \in \Q[t] \setminus \{0\}$ whence
all Hamiltonian quaternion polynomials (with exception of $0$) are motion
polynomials.

Evaluation of a rational motion over $\Q$ or $\DQ$ yields a rigid body
displacement for all parameter values $t \in \R \cup \{\infty\}$. For any point
$[x_0 + x_1\qi + x_2\qj + x_3\qk] = [x_0 + x] \in \Pj^3$, the polynomial
\begin{equation}
  \label{eq_trajectory}
  P(x_0 + x)\Cj{P} + 2x_0P\Cj{D}
\end{equation}
is a rational parametric expression. It represents a rational curve, the
\emph{trajectory} of $[x_0 + x]$ with respect to the rational motion $[P + \eps
D]$.

As a rational curve, any rational motion $[P+\eps D]$ is equipped with the
notion of a degree. But this degree is different from the usual meaning where
the degree of a rational motion is defined as the maximal (and also generic)
degree of any of its trajectories (c.\,f. \cite{juettler93,roeschel98}). We
adopt this latter convention. If we feel the need to distinguish between these
two different concepts of degrees we speak of \emph{quaternion degree}
(the degree of $[P+\eps D]$) and \emph{trajectory degree.}

The evaluation of a motion polynomial may yield a zero divisor for finitely many
parameter values $t_0 \in \K \cup \{ \infty \}$. It is convenient to assume that
this does not happen at $t_0 = \infty$, that is, the leading coefficient of $P +
\eps D$ is invertible. This is no loss of generality as we may apply a suitable
fractional linear parameter transformation. Multiplying (from the left or from
the right) with the inverse of the leading coefficient then yields a monic
motion polynomial. This amounts to a mere change of coordinates in the moving or
fixed coordinate frame of the rational motion under consideration. Once more,
this is no loss of generality in our context. Hence, we will feel free to assume
that a motion polynomial is monic whenever this seems appropriate.

\section{Exceptionally Low Degree of Trajectories}\label{chap2}

A motion polynomial $P + \eps D \in \DQ[t]$ of degree $n$ describes a rational
motion. We assume that $P + \eps D$ is reduced whence also the rational motion
is of quaternion degree $n$. A glance at Equation~\eqref{eq_trajectory} confirms
that its trajectories are generically of degree $2n$. But, as already mentioned
in the introduction, it is possible that this degree drops. Our aim in this
section is a characterization of reduced motion polynomials of degree $n$ that
parameterize rational motions of trajectory degree strictly less than~$2n$.

The trajectory of an arbitrary point $[x_0 + x] \in \Pj^3$ with $x = x_1\qi +
x_2 \qj + x_3 \qk$ is given by \eqref{eq_trajectory}. We may re-write this as
\begin{equation}\label{eq1}
  x_0\norm{P} + Px\Cj{P} + 2x_0P\Cj{D}
\end{equation}
and we assume, without loss of generality, that $P + \eps D$ is monic. Then the
degree of the rational parametric expression \eqref{eq1} equals $2n$ and the
degree of the corresponding rational motion is less than $2n$ precisely if the
expression in Equation~\eqref{eq1} has a real polynomial factor that is
independent from $x_0$ and $x$. A sufficient condition for this is existence of
a real polynomial factor of positive degree of the primal part $P$, i.\,e. $P = cQ$
with $c \in \R[t]$ and $Q \in \Q[t]$. Necessity of this condition is well-known
knowledge in the kinematics community but it is difficult to provide a precise
reference. Given the importance for this article, we provide a proof:

\begin{pro}
  \label{pro:degree}
  The degree of the rational curve parameterized by \eqref{eq1} is less than
  $2n$ for any choice of $x_0$ and $x$ if and only if $P$ has a real polynomial
  factor of positive degree.
\end{pro}

\begin{proof}
  We already argued for the sufficiency of this statement. In order to see
  necessity, assume that $\mrpf P = 1$ (``maximal real polynomial factor'') and
  consider the trajectory of the special point $[x]$ (i.\,e. $x_0 = 0$ and $x$
  is yet unspecified), given parametrically by $Px\Cj{P}$. Since it is a
  spherical curve, it is of even degree and so is its maximal real polynomial
  factor. Hence, we may assume that there exists a \emph{quadratic} real factor
  of $Px\Cj{P}$. But then, by \cite[Lemma~1]{li16}, there exists a linear right
  factor $t-p$ of $P$ such that $t-\Cj{p}$ is a left factor of $x\Cj{P}$. Now
  the following hold:
  \begin{itemize}
  \item There are at most $n$ different linear right factors of~$P$
    \cite{hegedus13},
  \item the linear polynomial $t-p$ is a right factor of $P$ if and only if
    $t-\Cj{p}$ is a left factor of $\Cj{P}$ (because of $\Cj{(AB)} = \Cj{B}\Cj{A}$
    for any quaternion polynomials $A$, $B \in \Q[t]$), and
  \item the linear polynomial $t-\Cj{q}$ is a left factor of $\Cj{P}$ if and
    only if $t-x\Cj{q}x^{-1}$ is a left factor of $x\Cj{P}$.
  \end{itemize}
  To prove the last statement, let us at first assume that $t-\Cj{q}$ is a left
  factor of $\Cj{P}$, i.\,e. there exists $H\in\Q[t]$ such that
  $\Cj{P}=(t-\Cj{q})H$. Then
  \begin{align*}
    x\Cj{P} =x(t-\Cj{q})H = (tx-x\Cj{q}x^{-1}x)H=(t-x\Cj{p}x^{-1})xH.
  \end{align*}
  Conversely, let us assume that there exists $G\in\Q[t]$ such that
  $x\Cj{P}=(t-x\Cj{q}x^{-1})G$. Then
  \begin{align*}
    \Cj{P}=x^{-1}x\Cj{P}=x^{-1}(t-x\Cj{q}x^{-1})G=(tx^{-1}-\Cj{q}x^{-1})G = (t-\Cj{q})x^{-1}G.
  \end{align*}

  We infer that the equation
  \begin{equation}
    \label{eq_halfturn}
    \Cj{p} = x\Cj{q}x^{-1}
  \end{equation}
  must be fulfilled for two linear right factors $t - p$ and $t - q$ of $P$.
  However, there are only finitely many linear right factors and the solution
  set of \eqref{eq_halfturn} is of positive co-dimension. Hence, for any choice
  of $x$ outside the union of the solution sets of the finitely many equations
  of type \eqref{eq_halfturn} where $t-p$ and $t-q$ are right factors of $P$
  ensures that the degree of $Px\Cj{P}$ equals $2n$. Hence, the assumption
  $\mrpf P = 1$ gives a contradiction.
\end{proof}

We denote the unique monic real polynomial factor of maximal degree of $P$ by $c
\coloneqq \mrpf P$. With this notation, the right hand side of \eqref{eq1}
becomes
\begin{align*}
  c(c x_0\norm{Q} + cQx\Cj{Q}+2x_0Q\Cj{D}).
\end{align*}
Obviously, this is a rational parametric expression of degree $2n - m$ where $m
\coloneqq \deg c$. A further degree reduction occurs if and only if $c$ and
$Q\Cj{D}$ have a common real polynomial factor of positive degree. This leads us
to the following definition:

\begin{defi}
  \label{def_exceptional_degree_reduction}
  Let $P + \eps D\in\DQ[t]$ be a reduced monic motion polynomial and set $c
  \coloneqq \mrpf P$. We say that a \emph{degree reduction by $m$} occurs if $c$
  is of degree $m$. Denote by $Q$ the unique polynomial in $\Q[t]$ such that $P
  = cQ$. We say that an \emph{exceptional degree reduction by $e$} occurs, if
  $c$ and $Q\Cj{D}$ have a real common factor of degree~$e$.
\end{defi}

To summarize, the maximal trajectory degree of a rational motion of degree $n$
with degree reduction $m$ equals $2n-m-e$ where $e$ is the degree of the real
gcd of $c$ and $Q\Cj{D}$.

\subsection{Algebraic Point of View}\label{sec2_1}

We continue with an alternative algebraic characterization for occurrence of an
exceptional degree reduction that allows the systematic construction of rational
motions with exceptionally low trajectory degree. In Section~\ref{sec2_2} it will be used
to derive a geometric criterion.

\begin{theo}
  \label{deg_red_alg}
  Let $P + \eps D$ be a reduced motion polynomial. Set $c \coloneqq
  \mrpf P$ and let $Q$ be such that $P = cQ$. Exceptional degree reduction
  occurs if and only if there exists a common right factor $H \in \Q[t]$ of $Q$
  and $D$ such that $\norm{H}$ divides $c$. In this case, the degree reduces
  exceptionally by $\deg\norm{H}$.
\end{theo}

\begin{proof}
  By Definition~\ref{def_exceptional_degree_reduction}, exceptional degree
  reduction occurs if $Q\Cj{D}$ shares a real polynomial factor of positive
  degree with $c$. By construction, $\gcd(c,Q) = 1$. Moreover, a common real
  factor of $c$ and $D$ is also a real factor of $P + \eps D$. But $P + \eps D$
  is reduced so that $\gcd(c,D) = 1$ and, by \cite[Proposition~2.1]{cheng16}, a
  real polynomial factor $h$ of $c$ and $Q\Cj{D}$ exists if and only if $Q$ and
  $D$ have a common right factor $H$ such that $\norm{H} = h$. It leads to the
  claimed exceptional degree reduction.
\end{proof}

\begin{rem}
  The degree of $\norm{H}$ is even whence exceptional degree reduction occurs
  always \emph{by an even number.}
\end{rem}

So far, only few examples of rational motions with exceptional degree
reduction appeared in literature. Among them are:
\begin{itemize}
\item The Cardan motion \cite[Chapter~9, \S~11]{bottema90} mentioned in
  Section~\ref{sec_introduction} is represented by a reduced rational parametric
  expression of degree three but its trajectories are only of degree two.
\item The same is true for the Darboux motion \cite[Chapter~9, \S~3]{bottema90}.
  In fact, a Darboux motion is the composition of a Cardan motion with a
  suitable oscillating translation.
\item Rational motions with trajectories of degree three have been described by
  Wunderlich in \cite{wunderlich84}. Their reduced motion polynomial
  representation is of degree four.
\end{itemize}
Computing the quaternion degrees of these motions is an easy exercise, see
e.\,g. \cite{li15} for the Darboux motion.

According to \cite{wunderlich84} the general form of a rational cubic motion is
obtained as composition of a Darboux motion with a suitable translational
motion. Our framework allows a simple re-interpretation of this result:

\begin{example}
  \label{ex_wunderlich}
  A Darboux motion can be written as $C \coloneqq \norm{Q}Q + \eps DQ$ with
  polynomials $Q$, $D \in \Q$, $\deg Q = 1$, $\deg D = 2$, c.\,f. \cite{li15}. A
  linear motion polynomial $F \coloneqq f + \eps G$ with $f \in \R[t]$ and $G
  \in \Q[t]$ describes a translation. Left-multiplying $C$ with $F$ yields $W
  \coloneqq FC = f\norm{Q}Q + \eps(fD+\norm{Q}G)Q$. The polynomial $W$ is of
  degree four so that one would expect trajectories of degree eight. But the
  degree of the maximal polynomial factor of its primal part equals $\deg f +
  \deg\norm{Q} = 1 + 2 = 3$ whence an ordinary degree reduction by three occurs.
  Moreover, the criterion of Theorem~\ref{deg_red_alg} ensures an exceptional
  degree reduction by $2\deg Q = \deg\norm{Q} = 2$. Thus, trajectories are
  indeed of degree~$8-3-2=3$, as claimed by Wunderlich.
\end{example}

Theorem~\ref{deg_red_alg} allows the systematic construction of rational motions
with exceptional degree reduction. Starting with a monic motion polynomial $R +
\eps E$, we pick an arbitrary polynomial $H \in \Q[t]$. The rational motion $cQ
+ \eps D$ with $c = H\Cj{H}$, $Q = RH$, and $D = EH$ then has an exceptional
degree reduction by $2\deg H$.

An even lower degree can be obtained by choosing $H$ as a right factor of $R$,
i.\,e. $R = R'H$ and then considering the polynomial $cR + \eps D$ or even $cR +
\eps FD$ where $F \in \Q[t]$ is a polynomial only subject to the condition $\deg
FD \le \deg cR$. Note however, that $cR + \eps FD$ is not generally a motion
polynomial, i.\,e. it violates the Study condition. Thus, one has to appeal to
the extension of the isomorphism \eqref{eq_isomorphism} to a homomorphism
between the group of invertible dual quaternions
\cite{pfurner16:_path_planning}. An alternative is, to restrict the construction
to planar motions. Here $cR$ is a linear combination of $1$ and $\qk$ and
$D$ is a linear combination of $\qi$ and $\qj$. If $F$ is also chosen as linear
combination of $1$ and $\qk$ it is guaranteed that $cR + \eps FD$ is a planar
motion polynomial.

Here is a concrete example of the construction above:

\begin{example}
  The line with the parametric expression $t-\qk+\varepsilon\qj$ on the Study
  quadric represents a planar rotation around a fixed axis different from $\qk$.
  Its primal part $P \coloneqq t-\qk$ has the norm $c \coloneqq t^2+1$.
  Multiplying the primal part with $c$ and the dual part with $P$ from the right
  yields the polynomial $c(t-\qk)+\varepsilon (t\qi+\qj)$. It parameterizes a
  planar motion of quaternion degree three with trajectories of degree two,
  hence a Cardan motion \cite{juettler93}.
\end{example}

\subsection{Geometrical Point of View}\label{sec2_2}

Now we complement the algebraic criterion of Theorem~\ref{deg_red_alg} for
exceptional degree reduction by a geometric interpretation.

\begin{theo}\label{deg_red_geom}
  Let $P + \eps D$ be a reduced motion polynomial. Set $c \coloneqq \mrpf P$ and
  let $Q$ be such that $P = cQ$. If there exists a quadratic real factor $(t -
  z)(t - \overline{z})$ with $z \in \C$ of $c$ such that $[Q(z)]$ and $[D(z)]$
  lie on a left ruling of the null quadric $\N$ in $\Pj(\CQ)$, there is an
  exceptional degree reduction by two. This condition is also necessary.
\end{theo}

\begin{rem}
  Note that the formulation of Theorem~\ref{deg_red_geom} holds for $[Q(z)] =
  [D(z)]$ as well. In this case, the two points lie on a left ruling but do not
  span it. Of course, the criterion of Theorem~\ref{deg_red_geom} also implies
  that $[Q(\overline{z})]$ and $[D(\overline{z})]$ lie on a left ruling.
\end{rem}

\begin{proof}[Proof of Theorem~\ref{deg_red_geom}]
  We assume at first that exceptional degree reduction by two occurs. By
  Theorem~\ref{deg_red_alg}, there exist polynomials $Q'$, $D'$, $H \in \Q[t]$
  such that $Q = Q'H$, $D = D'H$, $\deg H = 1$ and $\norm{H}$ divides $c$.
  Moreover, $H$ cannot be real because $Q$ has no real factors. Thus, there
  exists $z \in \C \setminus \R$ such that $\norm{H} = (t - z)(t -
  \overline{z})$. But then
  \begin{equation*}
    Q(z)\Cj{D}(z) = (Q\Cj{D})(z) = (Q'H\Cj{H}\Cj{D'})(z) =
    (H\Cj{H})(z) (Q'\Cj{D'})(z) = 0
\end{equation*}
  because $(H\Cj{H})(z) = \norm{H}(z) = 0$. Note that $Q(z)$ and $D(z)$ cannot
  vanish (because $P + \eps D$ is assumed to be reduced and $\mrpf Q = 1$) so
  that $[Q(z)]$ and $[D(z)]$ are well defined points. By
  Proposition~\ref{lem_rulings}, $[Q(z)] \vee [D(z)]$ lies on a left ruling.
  (This includes the case $[Q(z)] = [D(z)]$.)

  Now assume conversely that there exists $z \in \C$ such that $[Q(z)] \vee
  [D(z)]$ lies on a left ruling and $c' \coloneqq (t-z)(t-\overline{z}) \in
  \R[t]$ divides $c$. Proposition~\ref{lem_rulings} then implies $Q(z)\Cj{D}(z)
  = (Q\Cj{D})(z) = 0$. The same holds true for the complex conjugate
  $\overline{z}$ whence there exists $F \in \Q[t]$ such that $Q\Cj{D} = c'F$.
  Thus, exceptional degree reduction in the sense of
  Definition~\ref{def_exceptional_degree_reduction} occurs.
\end{proof}

\begin{rem}
  \label{rem_different_factors}
  From Definition~\ref{def_exceptional_degree_reduction} it is immediately clear
  that different quadratic factors $c_\ell \coloneqq
  (t-z_\ell)(t-\overline{z_\ell})$ give rise to different left rulings and
  contribute each to an exceptional degree reduction by two.
\end{rem}

It seems natural to interpret factors of multiplicity $\mu > 1$ of $c$ that lead
to an exceptional degree reduction in terms of left rulings and their
``multiplicity''. This we define as intersection multiplicity of the ruled
surface spanned by $Q$ and $D$ and the ruled surface of left rulings $\RLR$,
both viewed as rational curves on the so called Pl\"ucker or Klein quadric
\cite[Section~2.1]{pottmann10}, which provides a point model for lines in
three-dimensional projective space. This point model can be constructed by
identifying lines with half-turns about these lines. Half-turns are represented
by dual quaternions
$x + \eps y$ with zero scalar parts, i.\,e., $(x + \eps y) + (\Cj{x} + \eps
\Cj{y}) = 0$. Under this assumption, the Study condition $x\Cj{y} + y\Cj{x} = 0$
reduces to the Pl\"ucker condition \cite[Equation~(2.3)]{pottmann10} which is
the defining equation of the Pl\"ucker quadric. The line spanned by points
$[a]$, $[b] \in \Pj(\CQ)$ is represented by
\begin{equation}
  \label{eq_pluecker}
a\Cj{b} - \Cj{a}b + \eps (\Cj{a}b - b\Cj{a}).
\end{equation}
The ruled surface of left rulings $\RLR$ is given by all dual quaternions
\begin{equation}
  \label{eq_left_rulings}
  x - \eps x\quad\text{where}\quad
  x \in \CQ,\ x + \Cj{x} = x\Cj{x} = 0.
\end{equation}
It is a conic on the Pl\"ucker quadric. Equation~\eqref{eq_left_rulings} follows
from \eqref{eq_rulings} together with \eqref{eq_pluecker}: If the points $[a]$
and $[b]$ span a left ruling $\ell$, $a\Cj{b} = 0$ and $b\Cj{a} = 0$ hold,
whence $\ell$ is represented by $-\Cj{a}b + \eps \Cj{a}b$.

\begin{defi}
  For a rational curve $[C]$ the \emph{$m$-th osculating space} in $[C(t_0)]$
  is the projective subspace spanned by $[C(t_0)]$,
  $[\frac{\mathrm{d}}{\mathrm{d}t}C(t_0)]$, \ldots,
  $[\frac{\mathrm{d}^m}{\mathrm{d}t^m}C(t_0)]$. Two rational curves
  \emph{intersect with multiplicity $\mu$} in a point if their $m$-th osculating
  spaces at that point agree for $m = 0,1,\ldots,\mu-1$.
\end{defi}

Because $\RLR$ is a conic on the Pl\"ucker quadric, its $m$-th osculating
space for $m \ge 2$ is the conic's support plane. We denote it by $\gamma$. It is
described by \eqref{eq_left_rulings} but with the norm condition $x\Cj{x} = 0$
removed.

\begin{theo}
  \label{th_multiplicity}
  Let $P + \eps D$ be a reduced motion polynomial. Set $c \coloneqq \mrpf P$ and
  let $Q$ be such that $P = cQ$. If there exists a real factor $(t - z)^\mu(t -
  \overline{z})^\mu$ with $z \in \C$, $\mu \in \mathbb{N}$ of $c$ such that the
  ruled surface $[Q] \vee [D]$ intersects $\RLR$ in $[Q(z)] \vee [D(z)]$ with
  multiplicity $\mu$, there is an exceptional degree reduction by $2\mu$.
\end{theo}

\begin{proof}
  By \eqref{eq_pluecker} the ruled surface $[Q] \vee [D]$ is parameterized by
  the polynomial
  \begin{equation}
    \label{eq_ruled_surface}
Q\Cj{D} - \Cj{Q}D + \eps (\Cj{Q}D - D\Cj{Q}).
  \end{equation}
  It intersects $\RLR$ with multiplicity $\mu$ in $[Q(z)] \vee [D(z)]$, i.\,e.,
  their $m$-th osculating spaces coincide for $m = 0,1,\ldots,\mu-1$. For
  $\mu=1$ this means that $[Q(z)] \vee [D(z)]$ is a left ruling. For $\mu=2$
  this means that in addition the first derivative point lies on the conic
  tangent in $[Q(z)] \vee [D(z)]$. Because ruled surfaces are curves on a common
  quadric (the Pl\"ucker quadric), this is equivalent to the seemingly weaker
  condition that the derivative point lies in the conic's support plane
  $\gamma$. For $\mu \ge 3$ this means that also the higher order derivative
  points lie in $\gamma$. All in all, intersection with multiplicity $\mu$
  happens if all derivative points up to order $\mu-1$ lie in $\gamma$.

  The condition that the $m$-th derivative point at $t = z$ lies in the support
  plane of $\RLR$ (viewed as a conic on the Pl\"ucker quadric) is
  \begin{equation*}
    \frac{\mathrm{d}^m}{\mathrm{d}t^m}(Q\Cj{D} - \Cj{Q}D)(z) +
\frac{\mathrm{d}^m}{\mathrm{d}t^m}(\Cj{Q}D - D\Cj{Q})(z) = 0.
  \end{equation*}
  Using the Study condition $Q\Cj{D} + D\Cj{Q} = 0$, it reduces to
  \begin{equation*}
    \frac{\mathrm{d}^m}{\mathrm{d}t^m}(Q\Cj{D})(z) = 0.
  \end{equation*}
  Hence, intersection with multiplicity $\mu$ implies that $(t-z)^\mu$ is a
  factor of $Q\Cj{D}$. The same is true for $(t-\overline{z})^\mu$ and, by
  Definition~\ref{def_exceptional_degree_reduction}, exceptional degree
  reduction by $2\mu$ occurs.
\end{proof}

\begin{rem}
  From the proof of Theorem~\ref{th_multiplicity} it is easy to infer that an
  exceptional degree reduction by $2\mu$ caused by a real factor
  $(t-z)^\mu(t-\overline{z})^\mu$ of $c$ and $Q\Cj{D}$ implies intersection
  multiplicity $\mu$ of $\RLR$ and the ruled surface $[Q] \vee [D]$ at $[Q(z)]
  \vee [D(z)]$ (and also at $[Q(\overline{z})] \vee [D(\overline{z})]$). Similar
  to Remark~\ref{rem_different_factors} we can argue that common quadratic real
  factors of $c$ and $Q\Cj{D}$ and their multiplicities are in one-to-one
  correspondence with left rulings on $[Q] \vee [D]$ and their intersection
  multiplicities with~$\RLR$.
\end{rem}

The geometric characterization of exceptional degree reduction is capable of
explaining the observation that motion and inverse motion need not have
trajectories of the same degree. The inverse of a rational motion $[P+\eps D]$
is $[\Cj{P}+\eps \Cj{D}]$. This operation does not affect the degree of the
motion as rational curve on the Study quadric, but it interchanges the two
families of rulings on $\N$. Thus, it is possible that exceptional degree
reduction occurs for a rational motion while it does not for its inverse motion.

\begin{example}
  We consider the Darboux motion $C = \norm{Q}Q + \eps DQ$ of
  Example~\ref{ex_wunderlich} and denote the conjugate complex roots of the
  maximal real polynomial factor $\norm{Q}$ of the primal part by $z$ and
  $\overline{z}$. Evaluating reduced primal part and dual part at $z$ yields
  points $[Q(z)]$ and $[DQ(z)]$. We have $\norm{Q}(z)=0$ and hence also
  \begin{equation*}
    \norm{DQ(z)} = \norm{D(z)}\norm{Q(z)} = 0.
  \end{equation*}
  This means that the points $[Q(z)]$ and $[DQ(z)]$ lie on the null quadric $\N$
  and, by Corollary~\ref{cor2}, also lie on a left ruling. Since dual quaternion
  conjugation interchanges the two families of rulings on $\N$, the
  corresponding points of the inverse motion lie on a right ruling. Exceptional
  degree reduction for the inverse motion occurs if and only if they also lie on
  a left ruling, which is only possible if they coincide, i.\,e. $[Q(z)] =
  [DQ(z)]$. Darboux motions with this property are known as \emph{vertical
    Darboux motions.}
\end{example}

\section{Conclusion}

We gave precise algebraic and geometric criteria on rational motions with an
exceptional degree reduction, thus answering a question that has been open for a
while in the kinematics community. Our criteria (Theorem~\ref{deg_red_alg} and
Theorem~\ref{deg_red_geom}) are not invariant with respect to conjugation. This
explains the phenomenon that a motion and its inverse motion can be of different
trajectory degrees.

The geometric criterion for exceptional degree reduction can also be formulated
for algebraic motions (algebraic curves on the Study quadric). It is obvious to
conjecture that it characterizes exceptional degree reduction also in the
algebraic case. If this is the case, it would also explain the possibility of
different degrees of an algebraic motion and its inverse. This is a topic of future research.

\section*{Acknowledgments}

Johannes Siegele was supported by the Austrian Science Fund (FWF): P~30673 (Extended Kinematic Mappings and Application to Motion Design).
Daniel F. Scharler was supported by the Austrian Science Fund (FWF): P~31061 (The Algebra of Motions in 3-Space).
 
\bibliographystyle{plain}

\end{document}